\documentclass{amsart}
\usepackage{hyperref,longtable}
\usepackage{amsthm,amsmath,amssymb,tikz}
\usepackage{fancyhdr}
\def\checkmark{\tikz\fill[scale=0.4](0,.35) -- (.25,0) -- (1,.7) -- (.25,.15) -- cycle;} 
\usepackage[outer=1.5in, inner=1.25in, top=1.25in, bottom=1.25in]{geometry}
\usepackage[shortlabels]{enumitem}
\numberwithin{equation}{section}
\usepackage[english]{babel}


\begin{document}
	
	\newtheorem{thm}{Theorem}[section]
	\newtheorem{lem}[thm]{Lemma}
	\newtheorem{cor}[thm]{Corollary}
	\theoremstyle{definition}
	\newtheorem{hyp}[thm]{Hypothesis}
	\newtheorem{de}[thm]{Definition}
	\newtheorem{remark}[thm]{Remark}

	\title[Semiregular automorphisms]{On the order of semiregular automorphisms of cubic vertex-transitive graphs}
	
	\author[Barbieri]{Marco Barbieri}
	\address{Dipartimento di Matematica``Felice Casorati", University of Pavia, Via Ferrata 5, 27100 Pavia, Italy} 
	\email{marco.barbieri07@universitadipavia.it}
	
	\author[Grazian]{Valentina Grazian}
	\address{Dipartimento di Matematica e Applicazioni, University of Milano-Bicocca, Via Cozzi 55, 20125 Milano, Italy} 
	\email{valentina.grazian@unimib.it}
	
	\author[Spiga]{Pablo Spiga}
	\address{Dipartimento di Matematica e Applicazioni, University of Milano-Bicocca, Via Cozzi 55, 20125 Milano, Italy} 
	\email{pablo.spiga@unimib.it}
	
	\begin{abstract}
		We prove that, if $\Gamma$ is a finite connected cubic vertex-transitive graph, then either there exists a semiregular automorphism of $\Gamma$ of order at least $6$, or the number of vertices of $\Gamma$ is bounded above by an absolute constant.
	\end{abstract}	
	
	\maketitle	
	\section{Introduction}
	
	A fascinating old-standing question in the theory of group actions on graphs is the so-called \emph{Polycirculant Conjecture}: non-identity $2$-closed transitive permutation groups contain non-identity semiregular elements. This formulation of the conjecture was introduced by Klin \cite{Klin1998}. However, the question was previously posed independently by Maru\v{s}i\v{c} \cite[Problem~2.4]{Marusic1981} and Jordan \cite{Jordan1988} in terms of graphs: vertex-transitive graphs having more than one vertex admit non-identity semiregular automorphisms.
 
	In this paper, we focus our attention on cubic graphs. In \cite{MarusicScappellatoSemiregular}, Marusi\v{c} and Scappellato proved that, each cubic vertex-transitive graph admits a non-identity semiregular automorphism, settling the Polycirculant Conjecture for such graphs. Their proof did not take into account the order of the semiregular elements. In this direction, Cameron \emph{et al.} proved in \cite{CameronSheehanSpigaSemiregular} that, if $\Gamma$ is a cubic vertex-transitive graph, then $\mathrm{Aut}(\Gamma)$ contains a semiregular automorphism of order at least $4$. They also conjectured that, as the number of vertices of $\Gamma$ tends to infinity, the maximal order of a semiregular automorphism tends to infinity. This was proven false by the third author in \cite{SpigaSemiregularBurnside} by building a family of cubic vertex-transitive graphs where such a maximum is precisely $6$. In the light of these results, it is unclear whether $6$ is  optimal in the sense of minimizing the maximal order of a semiregular element. Broadly speaking, we are interested in 
	
\begin{align}\label{eq:eq1}\liminf_{\substack{|V\Gamma|\to \infty\\\Gamma \textrm{ cubic vertex-transitive}}}\max\{o(g)\mid g\in \mathrm{Aut}(\Gamma), g \textrm{  semiregular}\},
\end{align}
where we denote by $o(g)$ the order of the group element $g$.

	\begin{thm}\label{thm.main:1.1}
		The value of~\eqref{eq:eq1} is $6$.
	\end{thm}
	
	Theorem \ref{thm.main:1.1} is a consequence of the following result and the main result in~\cite{SpigaSemiregularBurnside}.
	\begin{thm}\label{thm.main:1.2}
		Let $(\Gamma,G)$ be a pair such that $\Gamma$ is a connected cubic graph and $G$ is
	 a subgroup of the automorphism group of $\Gamma$ acting vertex-transitively on $V\Gamma$. Then either $G$ contains a semiregular automorphism of order at least $6$ or the pair $(\Gamma,G)$ appears in Table~$\ref{table:table}$.
	\end{thm}
	
There is a considerable amount of work into the proof of Theorem~\ref{thm.main:1.2}. Broadly speaking, the proof divides into two main cases. In the first main case, the exponent of the group $G$ is very small, bounded above by $5$, and we use explicit knowledge on the finite groups having exponent at most $5$. The second main case is concerned with graphs admitting a normal quotient which is a cycle. Here, we need to refine our knowledge on the ubiquitous Praeger-Xu graphs and on the splitting and merging operators between cubic vertex-transitive graphs and $4$-valent arc-transitive graphs defined in~\cite{PSV.cubicCensus}.

\begin{remark}\label{rem}{\rm
		The veracity of Theorem~\ref{thm.main:1.2} for graphs with at most $1\,280$ vertices has been proven computationally using the database of small cubic vertex-transitive graphs  in~\cite{PSV.cubicCensus}. Therefore, in the course of the proof of Theorem~\ref{thm.main:1.2} whenever we reduce to a graph having at most $1\,280$ vertices we simply refer to this computation.}
\end{remark}
Table~\ref{table:table} consists of six columns. In the first column, we report the number of vertices of the exceptional cubic vertex-transitive graph $\Gamma$. In the second column, we report the order of the transitive subgroups $G$ of $\mathrm{Aut}(\Gamma)$ with $G$ not containing semiregular elements of order at least $6$: each subgroup is reported up to $\mathrm{Aut}(\Gamma)$-conjugacy class. In the third column, we report the cardinality of $\mathrm{Aut}(\Gamma)$. In the forth column, when $|\mathrm{V}\Gamma|\le 1\,280$, we report the number of the graph in the database of small cubic vertex-transitive graphs in~\cite{PSV.cubicCensus}. In the fifth column of Table~\ref{table:table}, we write the symbol $\checkmark$ when the graph is arc-transitive and the symbol $\dag$ when the graph is a split Praeger-Xu graph (see Section~\ref{sPX} for the definition of split Praeger-Xu graphs). Split Praeger-Xu graphs play an important role in our investigation and hence we are keeping track of this information in the forth column.
In the sixth column, for the graphs not appearing in the database of small cubic vertex-transitive graphs, we give as much information as possible.
	
\begin{longtable}{| p{.06\textwidth} | p{.24\textwidth} | p{.085\textwidth} |p{.05\textwidth}|p{.06\textwidth}|p{.355\textwidth}|} 
\hline
$|\mathrm{V}\Gamma|$&  $|G|$& $|\mathrm{Aut}(\Gamma)|$ &  DB &  $\checkmark\;/\;\dag$&  Comments\\\hline
4&    4, 4, 8, 12, 24 & 24 &   1&  $\checkmark$&  \\ \hline
6&    6 & 12 &  1&  & \\
&    6, 36 & 24 &   2&  $\checkmark$&  \\ \hline
8&   8 & 16&  1&  & \\
&    8, 8, 8, 8, 16, 24, 24, 48 & 48 &   2&  $\checkmark$&  \\ \hline
10&    10 & 20 &   1&  &\\
&    10 & 20 &  2&  &\\
&    20, 60, 120 & 120 &   3&  $\checkmark$&  \\ \hline
12&    12, 24 & 24 &  2&  & \\
 &   24, 24 & 48 &  4& $\dag$ &\\ \hline
16&   16, 16, 32, 32, 64, 64 & 128 &  2&$\dag$ &  \\
 &  16 & 32 &  3&  & \\
 &   16, 48 & 96 &  4&  $\checkmark$&  \\ \hline
18&  18, 108 & 216 &  4&  $\checkmark$&  \\
 &  36 & 72 &  5&  &\\ \hline
20&    20 & 20 &   2& & \\
 &  160, 160 & 320 &  3&$\dag$ & \\
 &  60 & 120 &  6&  $\checkmark$&  \\
 &  120 & 240 &  7&  $\checkmark$&  \\ \hline
24&   24 & 144 &  2&  $\checkmark$&  \\
 &  24 & 48 &  8&  &\\
 &  24 & 24 &  9&  &\\
 &  24 & 48 &  10& & \\
 &  24, 24 & 48 &  11&$\dag$  & \\ \hline
30&   720 & 1\,440 &  8&  $\checkmark$&  \\
&  60, 120 & 120 &  9&  & \\
 &   60 & 60 &  10&  & \\ \hline
32&    32 & 64 &  2&  & \\
&    32, 32, 64, 64 & 128 &  3& $\dag$&  \\
&    32, 96 & 192 &  4&  $\checkmark$&  \\ \hline
36&    36 & 72 &  9&  & \\ \hline
40&    160, 160 & 320 &  12&$\dag$  & \\ \hline
50&   100 & 200 &  7& & \\ 
 &   50, 150 & 300 &  8&  $\checkmark$&  \\ \hline
54&    108 & 216 &  11&  &\\ \hline
60&    60  & 360 & 2&  $\checkmark$&  \\
 &   60, 120 & 120 &  3&  &\\
 &   60 & 60 &  4&  &\\
 &   60 & 120 &  5&  &\\
 &   60 & 120 &  6&  &\\
 &   60, 120& 120 &  7&  &\\
 &   60 & 120 &  8&  &\\
 &   60 & 120 &  9&  &\\
 &   60  & 120 & 10&  &\\ \hline
64&    64, 192 & 384 &  2&  $\checkmark$&  \\
 &   64 & 256 &  4&  &\\
&   64, 64 & 128 &  11& $\dag$ &\\ \hline
80&    80, 160 &160 &  29&  &\\
 &   160, 160 &320 &  31&$\dag$  &\\ \hline
90&   720 &1\,440 &  20&  &\\ \hline
96&    96 &192  & 37&  &\\ \hline
100&    100 &200 &  19&  &\\ \hline
128&  128&256 &  5 &  &\\ \hline
     160& 160 &160 &  89&  &\\ 
     &  160 &160 & 90&  &\\ 
     &  160 &320 & 91 &  &\\
     &  160 &320 & 92 &  &\\
     &  160 &320 & 93 &$\dag$  &\\
     &  160 &320 & 94 &  &\\ \hline
     180&  720 & & 77&  &  \\ 
     &  360, 720& &  78&  &  \\ \hline
     250&  500& &  31&  &\\ \hline 
     256&  256, 768& &  30&  $\checkmark$&  \\ \hline 
     360&  360&720 &  176&  &\\ 
     &  360&720 &  177&  &\\ 
    &  360&720 &  178 &  &\\
    &  360&360 &  179 &  &\\
     &  360&720 &  180 &  &\\
     &  360&360 &  181 &  &\\
     &  360&720 &  182 &  &\\
     &  360&720 &  183 & & \\
    &  360&720  & 184 &  &\\
     &  360&720  & 185 &  &\\
     &  720&1\,440  & 268 &  & \\
     &  720&1\,440 &  270 &   & \\ \hline
     512&  512&1\,024  &  734& & \\ \hline
     810&  1\,620&1\,620 &  198&  & \\  \hline
      1\,024&  1\,024, 3\,072&6\,144 &  3\,470 &  $\checkmark$ & \\ \hline
      1\,250&  2\,500 & 2\,500 &  187 &  &  \\ \hline
      1\,280&  1\,280&2\,500 & 2\,591&  & \\ \hline
2\,560&  2\,560&5\,120  & &  & \\
\hline
6\,250 & 12\,500 & 25\,000 &&&covers of the graph with $1\,250$\\
& 12\,500 & 12\,500 &&&vertices, there are 2 graphs\\
\hline
31\,250 & 62\,500 & 125\,000 &&&covers of the graphs \\
  & 62\,500 & 125\,000 &&&with $6\,250$ vertices,\\
  & 62\,500 & 125\,000 &&&there are five graphs\\
  & 62\,500 & 62\,500 &&&\\
  & 62\,500 & 62\,500 &&&\\
\hline

65\,610 & 131\,220 &?&&&cover of the graph with $810$ vertices, only one graph\\
\hline
2$\cdot$ 5$^\ell$&4$\cdot$ 5$^\ell$&&&&$7\le \ell\le 34$\\
\hline
\caption{Exceptional cases for Theorem~\ref{thm.main:1.2}}\label{table:table}
\end{longtable}

\section{Main ingredients}\label{sec2}
\subsection{Permutations}
A permutation on the set $\Omega$ is a \emph{derangement} if it fixes no elements in $\Omega$. A permutation is \emph{semiregular} if all of its cycles have the same length. For instance, any derangement of prime order is semiregular. A permutation group $G$ on $\Omega$ is said to be \emph{transitive} if it has a single orbit on $\Omega$, and \emph{semiregular} if the identity is the only element fixing some points. If $G$ is both semiregular and transitive on $\Omega$, then $G$ is \emph{regular} on $\Omega$. Given a permutation group $G$, and an element $\alpha\in \Omega$, we denote by $\alpha^G$ the orbit of $\alpha$ under the action of $G$.

\begin{lem}\label{lem:2.1}
	Let $G$ be a permutation group on $\Omega$, and let $p$ be a prime. If all the elements of $G$ of order $p$ are derangements, then all $p$-elements of $G$ are semiregular.
\end{lem}
\begin{proof}
	Let $g\in G$ be an element of order $p^k$, for some positive integer $k$. Aiming for a contradiction, assume that $g$ is not semiregular, that is, there exists $\alpha\in \Omega$ such that $|\alpha^{\langle g \rangle}|\leq p^{k-1}$. Hence $g^{p^{k-1}}$ fixes $\alpha$, which implies $g^{p^{k-1}}$ is not a derangement, a contradiction.
\end{proof}

\begin{lem}\label{lem:2.2}
	Let $G$ be a permutation group acting on $\Omega$, and let $p$ and $q$ be two distinct primes. If $G$ has a semiregular element $g$ of order $p$ and a semiregular element $h$ of order $q$ with $gh=hg$,
	then $gh$ is a semiregular element of order $pq$.
\end{lem}
\begin{proof}
	Since $gh=hg$, $o(gh)=pq$ and hence it remains to prove that $gh$ is semiregular. Note that $(gh)^p=h^p$ is semiregular, and also $(gh)^q=g^q$ is semiregular. Therefore, each orbit of $\langle gh\rangle$ has size $pq$, proving that $gh$ is semiregular.
\end{proof}

\subsection{Graphs}	
A \emph{digraph} is a binary relation
$		\Gamma=(V\Gamma,A\Gamma),$
where $A\Gamma\subseteq V\Gamma\times V\Gamma$. We refer to the elements of $V\Gamma$ as \emph{vertices} and to the elements of $A\Gamma$ as \emph{arcs}. In this paper, a \emph{graph} is a finite simple undirected graph, that is, a pair
$		\Gamma=(V\Gamma,E\Gamma),$
where $V\Gamma$ is a set of vertices, and $E\Gamma$ is a set of unordered pairs of $V\Gamma$, called \emph{edges}. In particular, a graph can be thought of as a digraph where the binary relation is symmetric and anti-reflexive.

The \emph{valency} of a vertex $\alpha \in V\Gamma$ is  the number of edges containing $\alpha$. A graph is said to be \emph{cubic} when all of its vertices have valency $3$. A connected graph is a \emph{cycle} when  all of its vertices have valency $2$.

Let $\Gamma$ be a graph, and let $G$ be a subgroup of the automorphism group $\mathrm{Aut}(\Gamma)$ of $\Gamma$. If $G$ is transitive on $V\Gamma$, we say that $G$ is \emph{vertex-transitive}, similarly, if $G$ is transitive on $A\Gamma$, we say that $G$ is \emph{arc-transitive}. Moreover, $\Gamma$ is vertex- or arc-transitive when $\mathrm{Aut}(\Gamma)$ is vertex- or arc-transitive.

Let $\alpha, \beta \in V\Gamma$ be two adjacent vertices. We denote by $G_\alpha$ the \emph{stabilizer} of the vertex $\alpha$, by $G_{\{\alpha, \beta\}}$ the setwise stabilizer of the edge $\{\alpha, \beta\}$, by $G_{\alpha \beta}$ the pointwise stabilizer of the edge $\{\alpha, \beta\}$ (that is, the stabilizer of the arc $(\alpha, \beta)$ underlying the edge $\{\alpha, \beta\}$).

Let $\Gamma$ be a graph, and let $N\leq \mathrm{Aut}(\Gamma)$. The \emph{normal quotient} $\Gamma/N$ is the graph whose vertices are the $N$-orbits of $V\Gamma$, and two $N$-orbits $\alpha^N$ and $\beta^N$ are adjacent if there exists an edge $\{\alpha',\beta'\}\in E\Gamma$ such that $\alpha'\in \alpha^N$ and $\beta'\in \beta^N$. Note that  the valency of $\Gamma/N$ is at most the valency of $\Gamma$, and that, whenever $\Gamma$ is conneted, so is $\Gamma/N$. Furthermore, if the group $N$ is normal in some $G\leq \mathrm{Aut}(\Gamma)$, then $G/N$ acts (possibly unfaithfully) on $\Gamma/N$. If the group $G$ acts vertex- or arc-transitively on $\Gamma$, then $G/N$ has the same property on $\Gamma/N$.

The following result is inspired by an analogous result for $4$-valent graphs in~\cite[Lemma~1.13]{PS.fixedVertex}. 

\begin{lem}\label{lemma:auxiliary}
	Let $\Gamma$ be a connected cubic graph, let $\alpha$ be a vertex of $\Gamma$, let $G$ be a vertex-transitive subgroup of $\mathrm{Aut}(\Gamma)$ and let $N$ be a semiregular normal
	subgroup of $G$. Suppose $G_\alpha$ is a non-identity $2$-group and that the normal quotient $\Gamma/N$ is a cycle of length $r \ge 3$, and denote by $K$ the kernel of the
	action of $G$ on the $N$-orbits on $V\Gamma$. Then either
\begin{enumerate}
\item\label{martedi1} $G_\alpha$ has order $2$ and $|K_\alpha|=1$, or
\item\label{martedi2} $r$ is even and $G_\alpha=K_\alpha$ is an elementary abelian $2$-group of order at most $2^{r/2}$.
\end{enumerate}
\end{lem}
\begin{proof}
Let $\Delta_0 , \Delta_1 , \ldots , \Delta_{r-1}$ be the orbits of $N$ in its action on $V\Gamma$.
Since $\Gamma/N$ is a cycle, we may assume that $\Delta_i$ is adjacent to $\Delta_{i-1}$ and $\Delta_{i+1}$ with indices computed modulo $r$. Moreover, without loss of generality, we suppose that $\alpha\in \Delta_0$.

As $G_\alpha$ is a non-identity $2$-group, by a connectedness argument, $G_\alpha$ induces a group of order $2$ in its action on the neighbourhood of $\alpha$. In particular, $G_\alpha$ fixes a unique neighbour of $\alpha$. As usual, for each $\beta\in V\Gamma$, let $\beta'$ be the unique neighbour of $\beta$ fixed by $G_\beta$. 

Suppose that $\{\alpha,\alpha'\}$ is contained in an $N$-orbit. Since $\alpha\in \Delta_0$, we deduce $\alpha'\in\Delta_0$. Let $\beta$ and $\gamma$ be the other two neighbours of $\alpha$. As $\Gamma/N$ is a cycle of length $r\ge 3$, we have $\beta\in \Delta_1$ and $\gamma\in\Delta_{r-1}$. Since $\mathrm{Aut}(\Gamma/N)$ is a dihedral group of order $2r$ and since $G_\alpha$ contains an element swapping $\beta$ and $\gamma$, we deduce $|G_\alpha:K_\alpha|=2$. Now, $K_\alpha$ fixes by definition each $N$-orbit and hence it fixes setwise $\Delta_1$ and $\Delta_{r-1}$. Therefore, $K_\alpha$ fixes $\beta$ and $\gamma$, because $\beta$ is the unique neighbour of $\alpha$ in $\Delta_1$ and $\gamma$ is the unique neighbour of $\alpha$ in $\Delta_{r-1}$. This shows that $K_\alpha$ fixes pointwise the neighbourhood of $\alpha$; now, a connectedness argument shows that $K_\alpha=1$. In particular, part~\eqref{martedi1} is satisfied. For the rest of the argument, we suppose that $\{\alpha,\alpha'\}$ is not contained in an $N$-orbit.

This means that $\alpha$ has two neighbours in an $N$-orbit, say $\Delta_1$, and only one neighbour in the other $N$-orbit, say $\Delta_{r-1}$. (Thus $\alpha'\in\Delta_{r-1}$ and $\beta,\gamma\in \Delta_1$.)
 This implies that $r$ is even and, for every $i\in\{0,\ldots,r/2-1\}$, each vertex in $\Delta_{2i}$ has two neighbours in $\Delta_{2i+1}$ and only one neighbour in $\Delta_{2i-1}$. Therefore, $G/K$ is a dihedral group of order $r$ when $r\ge 8$ and  $G/K$ is elementary abelian of order $4$ when $r=4.$ Morever, $G/K$ acts regularly on $\Gamma/N$ and hence $G_\alpha=K_\alpha$. It remains to show that $K_\alpha$ is an elementary abelian $2$-group of order at most $2^r$.

Since $N$ is normal in $G$, the orbits of $N$ on the edge-set $E\Gamma$ form a $G$-
invariant partition of $E\Gamma$. We claim that, no two edges incident to a fixed
vertex of $\Gamma$ belong to the same $N$-edge-orbit. We argue by contradiction and
we suppose that $\alpha$ has two distinct neighbours $v$
and $w$ such that the edges $\{\alpha, v\}$ and $\{\alpha, w\}$ are in the same $N$-edge-orbit. In particular, there exists $n\in N$ with $\{\alpha,v\}^n=\{\alpha,w\}$. This gives $\alpha^n=\alpha$ and $v^n=w$, or $\alpha^n=w$ and $v^n=\alpha$. Since there are no edges inside an $N$-orbit, we cannot have $\alpha^n=w$ and $v^n=\alpha$. Therefore, $\alpha^n=\alpha$ and $v^n=w$. Since $N$ acts semiregularly on $V\Gamma$, we have $n=1$ and hence $v=v^n=w$, which is a contradiction.

Since $G$ is vertex-transitive,  the edges between $\Delta_{2i}$
and $\Delta_{2i+1}$ are partitioned into precisely two $N$-edge-orbits, let’s call these
two orbits $\Theta_{2i}$ and $\Theta_{2i}'$; whereas the edges between $\Delta_{2i}$ and $\Delta_{2i-1}$ form one $N$-edge-orbit, which we call $\Theta_{2i}''$.	
	
An element of $K$ (fixing setwise the sets $\Delta_{2i}$ and $\Delta_{2i+1}$) can map an edge in $\Theta_{2i}$ only to an edge in $\Theta_{2i}$
or to an edge in $\Theta_{2i}'$. On the other hand, as $G_\alpha$ is not the identity group, for every vertex $v \in\Delta_{2i}$  there is an
element $g \in G_v$ which maps an edge of $\Theta_{2i}$ incident to $v$ to the edge of $\Theta_{2i}'$
incident to $v$; and this element $g$ is clearly an element of $K$, because $G/K$ acts semiregularly on $\Gamma/N$. This shows that
the orbits of $K$ on $E\Gamma$ are precisely the sets $\Theta_{2i}\cup\Theta_{2i}',\Theta_{2i}''$, $i\in \{0,\ldots,r/2-1\}$. In other words,
each orbit of the induced action of $K$ on the set $E\Gamma/N = \{e^N : e \in E\Gamma \}$ has
length at most $2$. Consequently, if $X$ denotes the kernel of the action of $K$ on $E\Gamma$,
then $K/X$ embeds into $\mathrm{Sym}(2)^{r/2}$ and is therefore an elementary abelian 2-
group of order at most $2^{r/2}$.	

Let us now show that $X = N$. Clearly, $N \le X$. Let $v \in \Delta_0$. Since $N$ is
transitive on $\Delta_0$, it follows that $X = NX_v$. Suppose that $X_v$ is non-trivial and
let $g$ be a non-trivial element of $X_v$. Further, let $w$ be a vertex which is closest
to $v$ among all the vertices not fixed by $g$, and let $v = v_0 \sim v_1 \sim \cdots \sim v_m = w$ be
a shortest path from $v$ to $w$. Then $v_{m-1}$ is fixed by $g$. Since $g$ fixes each $N$-edge-orbit setwise and since every vertex of $\Gamma$ is incident to at most one
edge in each $N$-edge-orbit, it follows that $g$ fixes all the neighbours of $v_{m-1}$,
thus also $v_m$. This contradicts our assumptions and proves that $X_v$ is a
trivial group, and hence that $X = N$.
\end{proof}

\subsection{Praeger-Xu graphs}\label{PX}

To introduce the infinite family of split Praeger-Xu graphs $\mathrm{sC}(r,s)$, we need two ingredients: the Praeger-Xu graphs and the splitting operation. This section is devoted to introduce the ubiquitous $4$-valent Praeger-Xu graphs $\mathrm{C}(r,s)$ and their automorphism group. This infinite family was originally defined in \cite{PraegerXu}, and it was studied in detail by Gardiner, Praeger and Xu in~\cite{PraegerXu,GardinerPraeger.charTetraSymGraphs}, and more recently in~\cite{JajcayPW.cycleStrPrXuGraphs}. Here, we introduce them through their directed counterparts defined in~\cite{MR1029176}.

Let $r$ be an integer, $r \ge 3$. Then $\vec{\mathrm{C}}(r,1)$ is the lexicographic product of a directed cycle of length $r$ with an edgeless graph on $2$ vertices. In other words,
$\mathrm{V}\vec{\mathrm{C}}(r,1)=\mathbb{Z}_r\times\mathbb{Z}_2$ with the out-neighbours of a vertex $(x, i)$ being $(x+1,0)$ and $(x+1,1)$. We will identify the $(s-1)$-arc
\[(x,\varepsilon_0)\sim(x+1,\varepsilon_1)\sim\ldots\sim(x+s-1,\varepsilon_{s-1})\]
with the pair $(x; k)$ where $k=\varepsilon_0\varepsilon_1\ldots\varepsilon_{s-1}$ is a string in $\mathbb{Z}_2$ of length $s$. For $s \ge 2$, let $\mathrm{V}\vec{\mathrm{C}}(r,s)$ be the set of all $(s-1)$-arcs of $\vec{\mathrm{C}}(r, 1)$, let $h$ be a string in $\mathbb{Z}_2$ of length $s-1$ and let $\varepsilon\in \mathbb{Z}_2$. The out-neighbours of $(x; \varepsilon h) \in \mathrm{V}\vec{\mathrm{C}}(r,s)$ are $(x+1; h0)$ and $(x+1;h1)$. The \emph{Praeger-Xu graph} $\mathrm{C}(r,s)$ is then defined as the underlying graph of $\vec{\mathrm{C}}(r,s)$. We have that $\mathrm{C}(r, s)$ is a connected $4$-valent graph with $r2^s$ vertices (see~\cite[Theorem~2.8]{MR1029176}).

Let us now discuss the automorphisms of the graphs $\mathrm{C}(r, s)$. Every automorphism of $\vec{\mathrm{C}}(r,1)$ ($\mathrm{C}(r, 1)$, respectively) acts naturally as an automorphism of $\vec{\mathrm{C}}(r, s)$ ($\mathrm{C}(r, s)$, respectively) for every $s \ge 2$.
For $i \in \mathbb{Z}_r$, let $\tau_i$ be the transposition on $\mathrm{V}\vec{\mathrm{C}}(r,1)$ swapping the vertices $(i,0)$
and $(i, 1)$ while fixing every other vertex. This is clearly an automorphism of $\vec{\mathrm{C}}(r,1)$, and thus also of $\vec{\mathrm{C}}(r,s)$ for $s\ge 2$. Let
\[K := \langle \tau_i \mid i \in \mathbb{Z}_r\rangle,\]
and observe that $K\cong  C_2^r$. Further, let $\rho$ and $\sigma$ be the permutations
on $\mathrm{V}\vec{\mathrm{C}}(r,1)$ defined by
\[(x, i)^\rho := (x + 1, i) \quad \hbox{and} \quad (x, i)^\sigma := (x, -i).\]
Then $\rho$ is an automorphism of $\vec{\mathrm{C}}(r, 1)$ or order $r$, and $\sigma$ is an involutory automorphism of $\mathrm{C}(r,1)$ (but not of $\vec{\mathrm{C}}(r,1)$). Observe that the group $\langle \rho, \sigma\rangle$ normalises $K$. Let
\[H := K\langle \rho, \sigma\rangle \quad \hbox{and} \quad H^+ := K\langle \rho\rangle.\]
Then, for every $r \ge 3$ and $s \ge 1$,
\[C_2 \mathrm{wr} D_r \cong H \le \mathrm{Aut}(\mathrm{C}(r, s)) \quad \textup{and} \quad C_2 \mathrm{wr} C_r \cong H^+ \le \mathrm{Aut}(\vec{\mathrm{C}}(r,s)).\]
Moreover, $H$ ($H^+$, respectively) acts arc-transitively on $\mathrm{C}(r,s)$ 
($\vec{\mathrm{C}}(r,s)$, respectively) whenever $1 \le s \le r-1$. With three exceptions,
the groups $H$ and $H^+$ are in fact the full automorphism groups of $\mathrm{C}(r, s)$ and
$\vec{\mathrm{C}}(r,s)$, respectively.

\begin{lem}[{{\cite[Theorem 2.13]{GardinerPraeger.charTetraSymGraphs} and \cite[Theorem 2.8]{MR1029176}}}] \label{aux1}
	The automorphism group of a directed Praeger-Xu graph is \[\mathrm{Aut}(\vec{\mathrm{C}}(r, s)) = H^+,\]
	and, if $r\ne 4$, the automorphism group of a Praeger-Xu graph is
	\[\mathrm{Aut}(\mathrm{C}(r, s)) = H.\]
	Moreover,
	\[| \mathrm{Aut}(\mathrm{C}(4, 1)) : H| = 9, \quad | \mathrm{Aut}(\mathrm{C}(4, 2)) : H| = 3\]
	\[\hbox{and} \quad | \mathrm{Aut}(\mathrm{C}(4, 3)) : H| = 2.\]
\end{lem}
The Praeger-Xu graphs also admit the following algebraic characterization.

\begin{lem}[{{\cite[Lemma~1.11]{PS.fixedVertex} or~\cite[Lemma~3.7]{BGS}}}]\label{lem:aux}
	Let $\Gamma$ be a finite connected $4$-valent graph, let $G$ be a vertex- and edge-transitive group of automorphisms of $\Gamma$, and let $N$ be a minimal normal subgroup of $G$.
	If $N$ is a $2$-group and $\Gamma/N$ is a cycle of length at least $3$, then $\Gamma$ is isomorphic to a Praeger-Xu
	graph $\mathrm{C}(r, s)$ for some positive integers $r\leq 3$ and $s\leq r-1$.
\end{lem}
For more details on Praeger-Xu graphs, we refer also to \cite{JajcayPW.cycleStrPrXuGraphs, MR4321645, barbieri_grazian_spiga_2022}.

\subsection{The splitting and merging operations}\label{split}
The operation of \emph{splitting} were introduced in \cite[Construction 11]{PSV.cubicCensus}. Let $\Delta$ be a $4$-valent graph, let $\mathcal{C}$ be a partition of $E\Delta$ into cycles.
By applying the splitting operation to the pair $(\Delta,\mathcal{C})$, we obtain the graph, denoted by $\mathrm{s}(\Delta,\mathcal{C})$, whose vertices are
\[ V\mathrm{s}(\Delta,\mathcal{C}) := \{(\alpha,C)\in V\Delta \times \mathcal{C} \mid \alpha \in VC \},\]
and such that two vertices $(\alpha,C)$ and $(\beta,D)$ are declared adjacent if either $C\neq D$ and $\alpha=\beta$, or $C=D$ and $\alpha$ and $\beta$ are adjacent in $\Delta$. Observe that, since $\Delta$ is $4$-valent, there are precisely $2$ cycles in $\mathcal{C}$ passing through $\alpha$, thus $\mathrm{s}(\Delta,\mathcal{C})$ is cubic and $|V\mathrm{s}(\Delta,\mathcal{C})|=2|V\Delta|$.

Notice that, for any $G\leq \mathrm{Aut}(\Delta)$ such that its action is $\mathcal{C}$-invariant, $G\leq \mathrm{Aut}(\mathrm{s}(\Delta,\mathcal{C}))$. Moreover, if $G$ is also arc-transitive on $\Delta$ (in particular, the action of $G_\alpha$ on the neighbourhood of $\alpha$ is either the Klein four group, or the cyclic group of order $4$, or the dihedral group of order $8$), then $G$ is vertex-transitive on $\mathrm{s}(\Delta,\mathcal{C})$. For any vertex $(\alpha,C)\in \mathrm{s}(\Delta,\mathcal{C})$,
\[G_{(\alpha,C)} = G_\alpha \cap G_{\{C\}},\]
where $G_{\{C\}}$ is the setwise stabilizer of the cycle $C$. In particular, whenever $G$ is arc-transitive on $\Delta$, as $G_\alpha$ switches the two cycles passing through $\alpha$, $|G_\alpha:G_{(\alpha,C)}|=2$.

Now, we introduce the tentative inverse of the splitting operator: the operation of \emph{merging} (see \cite[Construction~7]{PSV.cubicCensus}). Let $\Gamma$ be a connected cubic graph, and let $G\leq \mathrm{Aut}(\Gamma)$ be a vertex-transitive group such that the action of $G_\alpha$ on the neighbourhood of $\alpha$ is cyclic of order $2$. In particular, $G_\alpha$ is a non-identity $2$-group. Hence, $G_\alpha$ fixes a unique neighbour of $\alpha$, which we denote by $\alpha'$. Observe that $(\alpha')' = \alpha$ and $G_\alpha = G_{\alpha'}$. Thus, the set $\mathcal{M} := \{\{\alpha, \alpha' \} \mid \alpha \in V\Gamma\}$ is a complete matching of $\Gamma$, while the edges outside $\mathcal{M}$ form a 2-factor, which we denote by $\mathcal{F}$. The group $G$ in its action on $E\Gamma$ fixes setwise both $\mathcal{F}$ and $\mathcal{M}$, and acts transitively on the arcs of each of these two sets. Let $\Delta$ be the graph with vertex-set $\mathcal{M}$ and two vertices $e_1 , e_2 \in \mathcal{M}$ are declared adjacent if they are (as edges of $\Gamma$) at distance $1$ in $\Gamma$. We may also think of $\Delta$ as being obtained by contracting all the edges in $\mathcal{M}$. Let $\mathcal{C}$ be the decomposition of $E\Delta$ into cycles given by the connected components of the the 2-factor $\mathcal{F}$. The merging operation applied to the pair $(\Gamma,G)$ gives as a result the pair $(\Delta,\mathcal{C})$.

Two infinite families of cubic graph have degenerate merged graphs, namely the circular and M\"{o}bius ladders. For any $n\ge 3$, a \emph{circular ladder graph} is a graph isomorphic to the Cayley graph
\[\mathrm{Cay}(\mathbb{Z}_n\times \mathbb{Z}_2,\{(0,1),(1,0),(-1,0)\}),\]
and, for any $n\ge 2$, a \emph{M\"{o}bius ladder graph} is a graph isomorphic to the Cayley graph \[\mathrm{Cay}(\mathbb{Z}_{2n},\{1,-1,n\}).\]
Observe that we consider the complete graph on $4$ vertices to be a M\"{o}bius ladder graph.

\begin{lem}\label{ladder1}
	Let $\Lambda$ be a (circular or M\"{o}bius) ladder, and let $G\leq \mathrm{Aut}(\Lambda)$ be a vertex-transitive group. Then either $|V\Lambda|\leq 10$ or $G$ contains a semiregular element of order at least $6$.
\end{lem}
\begin{lem}\label{ladder2}
	Unless $\Lambda$ is isomorphic to the skeleton of the cube or the complete graph on $4$ vertices, the automorphism group of a (circular or M\"{o}bius) ladder $\Lambda$ contains $N\leq \mathrm{Aut}(\Lambda)$, a normal cyclic subgroup of order $2$, such that the normal quotient $\Lambda/N$ is a cycle.
\end{lem}

\begin{remark}\label{remBis}
	Let $\Gamma$ be a connected cubic graph that is neither a circular nor a M\"{o}bius ladder, and let $G\leq \mathrm{Aut}(\Gamma)$ be a vertex-transitive group such that the action of $G_\alpha$ on the neighbourhood of $\alpha$ is cyclic of order $2$. Then~\cite[Lemma~9 and Theorem~10]{PSV.cubicCensus} imply that the merging operator applied to the pair $(\Gamma,G)$ gives a pair $(\Delta,\mathcal{C})$ such that $\Delta$ is $4$-valent, and the action of $G$ on $\Delta$ is faithful, arc-transitive and $\mathcal{C}$-invariant. This result motivates the use of the word \emph{degenerate} when referring to the circular and M\"{o}bius ladders.
\end{remark}

In view of \cite[Theorem~12]{PSV.cubicCensus}, the merging operator is the right-inverse of the splitting one, or, more explicitly, unless $\Gamma$ is a (circular or M\"{o}bius) ladder, splitting a pair $(\Delta,\mathcal{C})$ obtained via the merging operation on $(\Gamma,G)$ results in the starting pair. For our purposes, we need to show that the merging operator is also the left-inverse of the splitting one.
\begin{thm}\label{ms=id}
	Let $\Delta$ be a $4$-valent graph, let $\mathcal{C}$ be a partition of $E\Delta$ into cycles, and let $G\leq \mathrm{Aut}(\Delta)$ be an arc-transitive and $\mathcal{C}$-invariant group. Then the merging operation can be applied to the pair $(\mathrm{s}(\Delta,\mathcal{C}), G)$ and it gives as a result $(\Delta, \mathcal{C})$.
\end{thm}
\begin{proof}
	Let $(\alpha,C)$ be a generic vertex of $\mathrm{s}(\Delta,\mathcal{C})$, let $D\in \mathcal{C}$ be the other cycle of the partition passing through $\alpha$, and let $\beta,\gamma\in V\Delta$ be the neighbours of $\alpha$ in $C$. Then, using the fact that $G$ is arc-transitive on $C$,
	\[(\alpha,D)^{G_{(\alpha,C)}}=\{(\alpha,D)\} \quad\hbox{and}\quad (\beta,C)^{G_{(\alpha,C)}}=(\gamma,C)^{G_{(\alpha,C)}}=\{(\beta,C), (\gamma,C)\}.\]
	Therefore, for any vertex $(\alpha,C)\in V\mathrm{s}(\Delta,\mathcal{C})$, $G_{(\alpha,C)}$ acts on the neighbourhood of $(\alpha,C)$ as a cyclic group of order $2$. Hence, we can apply the merging operation to the pair $(\mathrm{s}(\Delta,\mathcal{C}), G)$. Furthermore, we deduce that
	\[\mathcal{M}=\{\{(\alpha,C),(\alpha,D)\} \mid \alpha \in VC \cap VD \}\]
	is a complete matching for $(\mathrm{s}(\Delta,\mathcal{C}), G)$. Thus the connected components of the resulting $2$-factor $\mathcal{F}=E\mathrm{s}(\Delta,\mathcal{C}) \setminus \mathcal{M}$ can be identified with the cycles of $\mathcal{C}$. Now, consider the map defined as
	\[\theta: \mathcal{M}\rightarrow V\Delta, \, \{(\alpha,C),(\alpha,D)\} \mapsto \alpha.\]
	Since a generic vertex $\alpha \in V\Delta$ belongs to precisely two distinct cycles, $\theta$ is bijective. Moreover, $\beta$ is adjacent to $\alpha$ in $\Delta$ if, and only if, either $\{(\alpha,C),(\beta,C)\}$ or $\{(\alpha,D),(\beta,D)\}$ is an edge in $\mathrm{s}(\Delta,\mathcal{C})$. In particular, $\theta$ also induces the bijection
	\[ \hat{\theta}: \mathcal{F} \rightarrow E\Delta, \, \{(\alpha,C),(\beta,C)\} \mapsto \{\alpha, \beta\},\]
	which sends the connected components of $\mathcal{F}$ into disjoint cycles of $\mathcal{C}$. This shows that $\theta$ is a graph isomorphism between $\Delta$ and the $4$-valent graph obtained by merging the pair $(\mathrm{s}(\Delta,\mathcal{C}), G)$, and that the resulting cycle partition is isomorphic to $\mathcal{C}$.
\end{proof}
\begin{cor}\label{cor.ms=id}
	Let $\Delta$ be a $4$-valent graph, let $\mathcal{C}$ be a partition of $E\Delta$ into cycles, and let $G\leq \mathrm{Aut}(\Delta)$ be an arc-transitive and $\mathcal{C}$-invariant group (and so $G\leq \mathrm{Aut}(\mathrm{s}(\Delta,\mathcal{C}))$). Suppose that $G\leq A \leq \mathrm{Aut}(\mathrm{s}(\Delta,\mathcal{C}))$ is a vertex-transitive group such that, for any vertex $\alpha \in V\mathrm{s}(\Delta,\mathcal{C})$, the action of $A_\alpha$ on the neighbourhood of $\alpha$ is cyclic of order $2$, then $A\leq \mathrm{Aut}(\Delta)$.
\end{cor}
\begin{proof}
	Note that, as $G$ is a subgroup of $A$, the actions of $G$ and $A$ on the neighbourhood of any vertex $\alpha$ coincide. In particular, applying the merging operation to the pair $(\mathrm{s}(\Delta,\mathcal{C}),A)$ yields the same result as doing it on the pair $(\mathrm{s}(\Delta,\mathcal{C}),G)$, that is, by Theorem~\ref{ms=id}, in both cases we obtain $(\Delta,\mathcal{C})$. The result follows by Remark~\ref{remBis}.
\end{proof}

\subsection{Split Praeger-Xu graphs}\label{sPX} In this section, we bring together the information of Sections~\ref{PX} and~\ref{split} to define and study the split Praeger-Xu graphs.

All the partitions of the edge set of a Praeger-Xu graph into disjoint cycles were classified in \cite[Section~6]{JajcayPW.cycleStrPrXuGraphs}. Regardless of the choice of the parameters $r$ and $s$, there exists a decomposition into disjoint cycles of length $4$ of the form
\[(x; 0h)\sim (x+ 1; h0) \sim (x; 1h) \sim (x+ 1; h1)\]
for some $x\in \mathbb{Z}_r$, and for some string $h$ in $\mathbb{Z}_2$ of length $s-1$.  We denote this partition by $\mathcal{S}$. Moreover, observe that the only two neighbours of $(x; 0h)$ in the $K$-orbit containing $(x+ 1; h0)$ are $(x+ 1; h1)$ and $(x+ 1; h0)$, and similarly the only two neighbours of $(x+1; h0)$ in the $K$-orbit containing $(x; 0h)$ are $(x ; 1h)$ and $(x; 0h)$. Therefore, $\mathcal{S}$ is the unique decomposition such that each cycle intersects exactly two $K$-orbits.
\begin{de}\label{def:2.4}
	The \emph{split Praeger-Xu graph} $\mathrm{sC}(r,s)$ is the cubic graph obtained from the pair $(\mathrm{C}(r,s), \mathcal{S})$ by applying the splitting operation.
\end{de}
\begin{lem}\label{aux2}
	For some positive integers $r\ge 3$ and $s\le r-1$, the automorphism group of the split Praeger-Xu graph is
	\[\mathrm{Aut}(\mathrm{sC}(r,s))=H,\]
	and it acts transitively on $V\mathrm{sC}(r,s)$.
\end{lem}
\begin{proof}
	Note that $H$ acts on the set of $K$-orbits in $V\mathrm{C}(r,s)$, thus each automorphism of $H$ maps any cycle of $\mathcal{S}$ to a cycle intersecting exactly two $K$-orbits, that is, to an element of $\mathcal{S}$. Thus, $H$ is $\mathcal{S}$-invariant, and so $H\leq \mathrm{Aut}(\mathrm{sC}(r,s))$. We now show the opposite inclusion. Let $\alpha \in V\mathrm{sC}(r,s)$ be a generic vertex, aiming for a contradiction we suppose that $\mathrm{Aut}(\mathrm{sC}(r,s))_\alpha$ does not act on the neighbourhood of $\alpha$ as a cycle of order $2$. Let $\alpha',\beta,\gamma$ be the neighbours of $\alpha$ where $\alpha'$ is fixed by the action of $H_\alpha$, and let $\delta$ be the unique vertex at distance $1$ from both $\beta$ and $\gamma$. Since $H_\alpha \leq \mathrm{Aut}(\mathrm{sC}(r,s))_\alpha$, our hypothesis implies that there exists an element $g \in \mathrm{Aut}(\mathrm{sC}(r,s))_\alpha$ such that $\beta ^g = \alpha '$ and $\gamma^g=\gamma$. This yields a contradiction because $\delta^g$ is ill-defined: in fact there is no vertex of $\mathrm{sC}(r,s)$ at distance $1$ from both $\gamma^g$ and $\delta^g$. Recall that, from Lemma~\ref{aux1}, if $r \neq 4$, then $H= \mathrm{Aut}(\mathrm{C}(r,s))$, and so, by Corollary~\ref{cor.ms=id}, $\mathrm{Aut}(\mathrm{sC}(r,s))\leq H$. On the other hand, if $r=4$, observe that $H$ is vertex-transitive on $\mathrm{sC}(r,s)$ and $\mathrm{Aut}(\mathrm{sC}(r,s))_\alpha = H_\alpha$, hence the equality holds by Frattini's argument.
\end{proof}
\begin{lem}\label{ex:PX}
	Let $G$ be a vertex-transitive subgroup of $\mathrm{Aut}(\mathrm{sC}(r,s))$. Then either $G$ contains a semiregular element of order at least $6$, or $(\mathrm{sC}(r,s),G)$ is one of the examples in Table~$\ref{table:table}$ marked with the symbol $\dagger$.
\end{lem}
\begin{proof}
	From Lemma~\ref{aux2}, we have $G\le H=K\langle \rho,\sigma\rangle$. Observe that $G/G\cap K\cong \langle\rho,\sigma\rangle$, otherwise $G$ is not transitive on the vertices of the split graph $\mathrm{sC}(r,s)$. From this, it follows that $G=V\langle \rho f,\sigma g\rangle$, for some $f,g\in K$, where $V=G\cap K$. Since $\rho$ has order $r$, we get that
	\[\begin{split}
		(\rho f)^r&= \rho f \rho \ldots (\rho f \rho) f
		\\&= \rho f \rho \ldots (\rho^2 \rho^{-1}f \rho) f
		\\&= \rho f \rho \ldots \rho^2 f^\rho f
		\\&= \rho f \rho^{r-1} \ldots f^\rho f
		\\&= f^{\rho^{r-1}} \ldots f^\rho f
	\end{split}\]
	is an element of $V$. Since $V$ is an elementary abelian $2$-group, the element $\rho f$ has order either $r$ or $2r$. Recalling that $V\leq K$,
	\[(\rho f)^r=\prod_{i=0}^{r-1}\tau_i^{a_i}\]
	with $a_i\in \{0,1\}$. Furthermore,
	\[\begin{split}
		(\rho f)^r\rho&= \rho (f \rho \ldots \rho f \rho f \rho)
		\\&=\rho (f f^{\rho} \ldots f^{\rho^{r-2}} f^{\rho^{r-1}})
		\\&=\rho (f^{\rho^{r-1}} \ldots f^\rho f)
		\\&=\rho (\rho f)^r
	\end{split}\]
	thus $\rho$ centralizes $(\rho f)^r$. From this, and from the fact that $\langle \rho\rangle$ acts transitively on $\{\tau_0,\ldots,\tau_{r-1}\}$, we deduce that
	\[(\rho f)^r=\prod_{i=0}^{r-1}\tau_i^{a}\]
	where $a$ is either $0$ or $1$. If $a=0$, then $\rho f$ is a semiregular element of order $r$. In particular, either $r\ge 6$, or the number of vertices of $\mathrm{sC}(r,s)$ is $r2^s$, which is bounded by $5\cdot 2^5=160$, and we finish by Remark~\ref{rem}. On the other hand, if $a=1$, $\rho f$ has order $2r$, and it corresponds to the so-called \emph{super flip} of the Praeger-Xu graph $\mathrm{C}(r,s)$. Since $(\rho f)^r$ does not fix any vertex in $\mathrm{C}(r,s)$, and since the vertex-stabilizers for a split graph has index $2$ in the vertex-stabilizer of the starting graph, for any  vertex $\alpha\in V\mathrm{sC}(r,s)$, we obtain that $(\rho f)^r \notin G_\alpha$. Hence $\rho f$ is semiregular of order $2r\ge 6$.
\end{proof}

To conclude this section, we show a result mimicking Lemma~\ref{lem:aux} for cubic graphs.

\begin{lem}\label{lem:2.5}
	Let $\Gamma$ be a connected cubic vertex-transitive graph, let $G\le\mathrm{Aut}(\Gamma)$ be a vertex-transitive group such that the action of $G_\alpha$ on the neighbourhood of $\alpha$ is cyclic of order $2$, and let $N$ be a minimal normal subgroup of $G$. If $N$ is a $2$-group and $\Gamma/N$ is a cycle of length at least $3$, then $\Gamma$ is isomorphic either to a circular ladder, or to a M\"{o}bius ladder, or to $\mathrm{sC}(r,s)$, for some positive integers $r\ge 3$ and $s \le r-1$.
\end{lem}
\begin{proof}
	We already know by Lemma~\ref{ladder2} that both ladders admit a cyclic quotient graph, thus we can suppose that $\Gamma$ is not isomorphic to a circulant ladder or to a M\"{o}bius ladder. By hypothesis, we can apply the merging operator to $(\Gamma,G)$, obtaining the pair $(\Delta, \mathcal{C})$. Since we have excluded the possibility of $\Gamma$ being a ladder, by Remark~\ref{remBis}, $\Delta$ is $4$-valent, and the action of $G$ on $\Delta$ is faithful, arc-transitive and $\mathcal{C}$-invariant.
	Since the action of $N$ cannot map edges in $\mathcal{M}$ to edges in $\mathcal{F}$, the quotient graph $\Gamma/N$ retains a partition into two disjoint sets of edges, namely $\mathcal{M}/N$ and $\mathcal{F}/N$. Moreover, since $\mathcal{M}$ is a complete matching, each edge in $\mathcal{M}/N$ is adjacent to precisely two edges in $\mathcal{F}/N$, and vice versa. This implies that the edges of $\Delta/N$ coincide with the elements of $\mathcal{F}/N$, two of which are adjacent if they share the same neighbour in $\mathcal{M}/N$. If $r\ge 6$, then $\Delta /N$ is a cycle of length $r/2$. From Lemma~\ref{lem:aux}, we deduce that $\Delta$ is isomorphic to $\mathrm{C}(r,s)$, for some positive integers $r\ge 3$ and $s\le r-1$. Observe that, as $\mathcal{C}$ coincides with the connected components of $\mathcal{F}$, each cycle in $\mathcal{C}$ intersects precisely two $K$-orbits. This implies that $\mathcal{C}=\mathcal{S}$, and so~\cite[Theorem~12]{PSV.cubicCensus} yields that $\Gamma$ is isomorphic to \[\mathrm{s}(\Delta,\mathcal{C})=\mathrm{s}(\mathrm{C}(r,s),\mathcal{S})=\mathrm{sC}(r,s). \qedhere\]
	Now, suppose that $r=4$. In this case, we have that $G$ is a $2$-group, hence $|N|=2$ and $|V\Gamma|=8$, and so the only possibility is for $\Gamma$ to be a (cirular or M\"{o}bius) ladder, which we already excluded. 
\end{proof}

	\section{Proof of Theorem~\ref{thm.main:1.2}}

We aim to prove Theorem~\ref{thm.main:1.2} by contradiction. In this section we will assume the following.
\begin{hyp}\label{hyp:3.1}
	Let $\Gamma$ be a connected cubic graph, and let $G\le\mathrm{Aut}(\Gamma)$ such that the pair $(\Gamma, G)$ is a minimal counterexample to Theorem \ref{thm.main:1.2}, first with respect to the cardinality of $V\Gamma$, and then to the order of $G$. From Remark~\ref{rem}, we have $|V\Gamma|>1\,280$. Let $\alpha$ be an arbitrary vertex of $\Gamma$. Let $N$ be a minimal normal subgroup of $G$.
\end{hyp}

Since $\Gamma$ is connected, the stabilizer $G_\alpha$ is a $\{2,3\}$-group. More generally, if $\Delta$ is a connected $d$-regular graph, then no prime bigger than $d$ divides the order of a vertex stabilizer (this follows from an elementary connectedness argument, see for instance~\cite[Lemma~3.1]{SpigaSemiregularBurnside} or~\cite[Lemma~3.2]{MarusicScappellatoSemiregular}). Moreover, $G$ must be a $\{2,3,5\}$-group, otherwise we can find derangements of prime order at least $7$, hence semiregular elements.

Since $N$ is a minimal normal subgroup of $G$, $N$ is a direct product of simple groups, any two of which are isomorphic. Clearly, $N$ is a $\{2,3,5\}$-group, and $N_\alpha$ is a $\{2,3\}$-group. Thus $N$ is a direct product $S^l$, for some positive integer $l$ and for some simple $\{2,3,5\}$-group $S$. Using the Classification of Finite Simple Groups, we see that the collection of simple $\{2,3,5\}$-groups consists of
\[C_2,\, C_3,\, C_5,\, \mathrm{Alt}(5),\, \mathrm{Alt}(6),\, \mathrm{PSp}(4,3),\]
see for instance~\cite{MR338153}.

\begin{lem}\label{lem:3.2}
	Under Hypothesis~$\ref{hyp:3.1}$, if $N_\alpha$ is a $2$-group, then $N$ is an elementary abelian $p$-group, for some prime $p\in \{2,3,5\}$.
\end{lem}
\begin{proof}
	If $N$ is abelian, then there is nothing to prove. Thus, suppose that $N=S^l$, where $S\in \{\mathrm{Alt}(5), \mathrm{Alt}(6), \mathrm{PSp}(4,3)\}$ and $l\ge 1$.
	
	Assume $l\geq 2$. Let $S$ and $T$ be two distinct direct factors of $N$. Then $S_   \alpha$ and $T_\alpha$ are $2$-groups,  because so is $N_\alpha$. Thus, by Lemma~\ref{lem:2.1}, all the $3$- and $5$-elements of $S$ and $T$ are semiregular. Applying Lemma~\ref{lem:2.2}, we obtain that $S\times T$, contains a semiregular element of order $15$. Thus $G$ contains a semiregular element of order exceeding $6$, contradicting Hypothesis~\ref{hyp:3.1}.
	
	Assume $l=1$. If $N=\mathrm{PSp}(4,3)$, then Lemma~\ref{lem:2.1} implies that the $3$-elements in $N$ are semiregular. As $\mathrm{PSp}(4,3)$ contains elements of order $9$, $G$ contains a semiregular element of order $9$, contradicting Hypothesis~\ref{hyp:3.1}. Thus, $N$ is either $\mathrm{Alt}(5)$ or $\mathrm{Alt}(6)$.
	
	We claim that $G$ is almost simple, that is, $N$ is the unique minimal normal subgroup of $G$. Aiming for a contradiction, let $M$ be a minimal normal subgroup of $G$ distinct from $N$. If $\Gamma/M$ is a cubic graph, then $M_\alpha =1$, and hence each element of $M$ is semiregular. Since $[N,M]=1$, by Lemma~\ref{lem:2.2},  $G$ contains a semiregular element of order at least $10$, against Hypothesis~\ref{hyp:3.1}. On the other hand, suppose that $\Gamma/M$ is not cubic. Regardless of the valency of $\Gamma/M$, the group that $G$ induces in its action on the vertices of $\Gamma/M$ is a subgroup of a dihedral group, hence it is a soluble group. In particular, as $N$ is a non-abelian simple group, $N$ acts trivially on the vertices of $\Gamma/M$. This means that $N$ fixes setwise each $M$-orbit. If $M$ is abelian, then $M$ acts regularly on each of its orbits. However, as $N$ commutes with $M$ and fixes each $M$-orbit, this contradicts the fact that $N$ is non-abelian.\footnote{Recall that, if $X\le\mathrm{Sym}(\Omega)$ is an abelian group and $X$ acts regularly on $\Omega$, then $X={\bf C}_{\mathrm{Sym}(\Omega)}(X)$.}
	Therefore, $M$ is not abelian. 
	In particular, there is a prime $p\geq 5$ that divides the order of $M$, and the elements of $M$ of order $p$ are semiregular. As before, applying Lemma \ref{lem:2.2}, we get that $NM$ contains a semiregular element of order $3p$, a contradiction. We conclude that $N$ is the unique minimal normal subgroup of $G$.
	
	Notice that $\mathrm{Alt}(5)\le G\le\mathrm{Sym}(5)$ or $\mathrm{Alt}(6)\le G\le\mathrm{Aut}(\mathrm{Alt}(6))$. A computer computation in each of these cases shows that, if $G\le\mathrm{Aut}(\Gamma)$ has no semiregular elements of order at least $6$, then $|V\Gamma|\in \{30,60, 90,180,360\}$, which contradicts Hypothesis~\ref{hyp:3.1}.
\end{proof}

From here on, we divide the proof in five cases:
\begin{itemize}
	\item $G_\alpha=1$;
	\item $G_\alpha\ne 1$ and $N$ is transitive on $V\Gamma$;
	\item $G_\alpha\neq 1$ and $N$ has two orbits on $V\Gamma$;
	\item $G_\alpha\ne 1$ and $\Gamma/N$ is a cycle of length at least $3$;
	\item $G_\alpha\ne 1$ and  $\Gamma/N$ is a cubic graph.
\end{itemize}

\subsection{$G_\alpha=1$}\label{sec:0}In this case $\Gamma$ is a Cayley graph over $G$. This means that there exists an inverse-closed subset $I$ of $G$ with $\Gamma\cong\mathrm{Cay}(G,I)$. We recall that $\mathrm{Cay}(G,I)$ is the graph having vertex set $G$ where two vertices $x$ and $y$ are declared to be adjacent if and only if $yx^{-1}\in I$. Since $\Gamma$ has valency $3$, we have $|I|=3$. Moreover, since $\Gamma$ is connected, we have $G=\langle I\rangle$. In particular, $G$ is generated by at most $3$ elements. More precisely, either $I$ consists of three involutions or  $I$ consists of  an involution and an element of order greater than $2$ together with its inverse.

In what follows we say that a finite group $X$ satisfies $\mathcal{P}$ if $X$ is generated by either three involutions, or by an involution and by an element of order greater than $2$. In particular, $G$ satisfies $\mathcal{P}$.

Since each element of $G$ is semiregular and since $G$ has no semiregular elements of order at least $6$, we deduce that each element of $G$ has order at most $5$. As customary, we let $$\omega(G):=\{o(g)\mid g\in G\}$$ be the spectrum of $G$. Observe that $$\{1,2\}\subseteq \omega(G)\subseteq \{1,2,3,4,5\}.$$ Since $G$ is generated by at most $3$ elements, we deduce from Zelmanov's solution of the restricted Burnside problem that $|G|$ is bounded above by an absolute constant. We divide the proof depending on $\omega(G)$.

Assume $\omega(G)=\{1,2\}$. In this case, $G$ is elementary abelian and, since $G$ is generated by at most $3$ elements, we deduce $|G|\le 8$, which contradicts Hypothesis~\ref{hyp:3.1}.

Assume $\omega(G)=\{1,2,4\}$. Here, either  $G$ is generated by an element of order $2$ and an element of order $4$, or $G$ is generated by three involutions. We resolve these two cases with a computer computation. Suppose first that $G$ is generated by an involution and by an element of order $4$. We have  constructed the free group $F:=\langle x,y\rangle$ and we have constructed the set $W$ of words in $x,y$ of length at most $6$. Then, we have constructed the finitely presented group $\bar F:=\langle F|x^2,\{w^4:w\in W\}\rangle$. We use the ``bar'' notation for the projection of $F$ onto $\bar F$. Now, $\bar x$ has order $2$ and $\bar y$ has order $4$. Furthermore, each element of $\bar F$ that can be written as a word in $\bar x$ and $\bar y$ of length at most $6$ has order at most $4$. (The number $6$ was chosen arbitrarily but large enough to guarantee to get an upper limit on the cardinality of $G$.) A computer computation shows that $\bar F$ has order $64$ and  exponent $4$. This proves that the largest group of exponent $4$ and generated by an involution and by an element of order $4$ has order $64$. Now, $G$ is a quotient of $\bar F$ and hence $|G|\le |\bar F|\le 64$, which contradicts Hypothesis~\ref{hyp:3.1}. 
Next, suppose that $G$ is generated by three involutions. The argument here is very similar. We have considered the free group $F=\langle x,y,z\rangle,$ and we have considered the set $W$ of words in $x,y,z$ of length at most $6$. We have verified that $\langle F|x^2,y^2,z^2,\{w^4:w\in W\}\rangle$ has order  $1024$ and exponent $4$. This shows that $|G|\le 1\,024$, which contradicts Hypothesis~\ref{hyp:3.1}.

Assume $\omega(G)=\{1,2,3\}$. The groups having spectrum $\{1,2,3\}$ are classified in~\cite{MR1575072}. Routine computations in the list of groups $X$ classified in ~\cite[Theorem]{MR1575072} show that, if $X$ satisfies $\mathcal{P},$ then $|X|\le 18$, which contradicts Hypothesis~\ref{hyp:3.1}.

Assume $\omega(G)=\{1,2,5\}$. The groups having spectrum $\{1,2,5\}$ are classified in~\cite{MR562004}. As above, since $G$ satisfies $\mathcal{P}$, we deduce from a case-by-case analysis in the groups appearing in~\cite{MR562004} that $|G|\le 80$, which contradicts Hypothesis~\ref{hyp:3.1}.

Assume $\omega(G)=\{1,2,3,4\}$. The groups having spectrum $\{1,2,3,4\}$ are classified in~\cite{MR1132578}.  As above, since $G$ satisfies $\mathcal{P}$, we deduce from a case-by-case analysis in the groups appearing in~\cite[Theorem]{MR1132578} that $|G|\le 96$, which contradicts Hypothesis~\ref{hyp:3.1}.

Assume $\omega(G)=\{1,2,4,5\}$. The groups having spectrum $\{1,2,4,5\}$ are classified in~\cite{MR1711942}. This case is sligthly more involved and hence we do give more details. We have three cases to consider: 
\begin{enumerate}
	\item\label{case1} $G=T\rtimes D$ where $T$ is a non-trivial elementary abelian normal $2$-subgroup and $D$ is a non-abelian group of order $10$,
	\item\label{case2} $G=F\rtimes T$ where $F$ is an elementary abelian normal $5$-subgroup and $T$ is isomorphic to a subgroup of a quaternion group of order $8$,
	\item\label{case3} $G$ contains a normal $2$-subgroup $T$ which is nilpotent of class at most $6$ such that $G/T$ is a $5$-group.
\end{enumerate}

Suppose that~\eqref{case1} holds. Clearly, $D$ is the dihedral group of order $10$ and $T$ is a module for $D$ over the field $\mathbb{F}_2$ of cardinality $2$. The dihedral group $D$ has two irreducible modules over $\mathbb{F}_2$ up to equivalence: the trivial module and a $4$-dimensional module $W$. Since $G$ has no elements of order $10$, we deduce $V\cong W^\ell$, for some $\ell\ge 1$. We have verified with a computer computation that $W^3\rtimes D$ does not satisfy $\mathcal{P}$ and hence $G\cong W^\ell\rtimes D$ with $\ell\le 2$. We deduce that $|G|=|V\Gamma|\in \{10\cdot 16,10\cdot 16^2\}=\{160,2\,560\}$. From Hypothesis~\ref{hyp:3.1}, we have $|V\Gamma|>1\,280$ and hence $G\cong W^2\rtimes D$. We have constructed all connected cubic Cayley graphs over $W^2\rtimes D$ and we have found only one (up to isomorphism), therefore we obtain the example in Table~\ref{table:table}.

Suppose that~\eqref{case2} holds. Since $G$ satisfies $\mathcal{P}$, while the quaternion group of order $8$ does not, we deduce that $T$ is cyclic of order $4$. Thus $G=F\rtimes\langle x\rangle$, for some $x$ having order $4$. As $G$ satisfies $\mathcal{P}$, this means that $G=\langle x,y\rangle$, for some involution $y$. Clearly, $y=fx^2$ for some $f\in F$. As $G=\langle x,y\rangle=\langle x,fx^2\rangle=\langle x,f\rangle$, we have $F=\langle f,f^x,f^{x^2},f^{x^3}\rangle$.
Since $y=fx^2$ has order $2$ and $x$ has order $4$, we deduce 
$$1=y^2=fx^2fx^2=ff^{x^2},$$
that is,
$f^{x^2}=f^{-1}$. Now, $F=\langle f,f^x,f^{x^2},f^{x^3}\rangle=\langle f,f^x,f^{-1},(f^x)^{-1}\rangle=\langle f,f^x\rangle$. Thus $|F|\le 25$ and hence $|G|\le 100$, which contradicts Hypothesis~\ref{hyp:3.1}.

Suppose that~\eqref{case3} holds. Since $G$ satisfies $\mathcal{P}$, we deduce that $G/T$ is cyclic of order $5$. Thus $G=T\rtimes\langle x\rangle$, for some $x$ having order $5$. This means that $G=\langle x,y\rangle$, for some involution $y$. Clearly, $y\in T$. From Hypothesis~\ref{hyp:3.1}, we have  $|G|=|V\Gamma|> 1\, 280$.  Let $N$ be a minimal normal subgroup of $G$. We have $N\le T$ and $N$ is an irreducible $\mathbb{F}_2\langle x\rangle$-module. The cyclic group of order $5$ has two irreducible modules over $\mathbb{F}_2$ up to equivalence: the trivial module and a $4$-dimensional module. Since $G$ has no elements of order $10$, $x$ does not centralize $N$ and hence $N$ is the irreducible $4$-dimensional module for the cyclic group of order $5$. In particular, $|N|=2^4$. Consider $\bar{G}:=G/N$. Now, $$\{1,2,5\}\subseteq \omega(\bar{G})\subseteq \omega(G)=\{1,2,4,5\}.$$ Assume $\omega(\bar{G})=\{1,2,5\}$. From the discussion above (regarding the finite groups having spectrum $\{1,2,5\}$ and satisfying $\mathcal{P}$), we have $|\bar{G}|\le 80$ and hence $|G|=|G:N||N|\le 80\cdot 16=1\,280$, which is a contradiction. Therefore, $\omega(\bar{G})=\{1,2,4,5\}$. Since $(\Gamma,G)$ was chosen minimal in Hypothesis~\ref{hyp:3.1}, we have $|\bar{G}|\le 1\,280$. Therefore $(\Gamma/N,\bar{G})$ appears in Table~\ref{table:table}. An inspection on the groups appearing in this table shows that there is only one group having spectrum $\{1,2,4,5\}$ and is the group of order $1\,280$. Thus we know precisely $\bar{G}$.
Now, the group $G$ is an extension of $\bar{G}$ by $N$ and hence it can be computed with the cohomology package in the computer algebra system magma. We have computed all the extensions $E$ of $\bar{G}$ via $N$ and we have verified that none of the extensions $E$ has the property that $\omega(E)=\{1,2,4,5\}$ and with $E$ satisfying $\mathcal{P}$.

Assume $\omega(G)=\{1,2,3,5\}$. The groups having spectrum $\{1,2,3,5\}$ are classified in~\cite{Mazurov1}. We deduce from~\cite{Mazurov1} that $G\cong A_5$, which contradicts Hypothesis~\ref{hyp:3.1}.

Assume $\omega(G)=\{1,2,3,4,5\}$. The groups having spectrum $\{1,2,3,4,5\}$ are classified in~\cite{MR1132578}. We deduce from~\cite[Theorem]{MR1132578} that either $G\cong A_6$ or $G\cong V^\ell\rtimes A_5$ where $V$ is a $4$-dimensional natural module over the finite field of size $2$ for $A_5\cong \mathrm{SL}_2(4)$ and $\ell\ge 1$. The group $V^2\rtimes A_5$ does not satisfy $\mathcal{P}$ (this can be verified with a computer computation). Therefore, either $G\cong A_6$ or $G\cong V\rtimes A_5$. Thus $|G|=|V\Gamma|\le 960$, which contradicts Hypothesis~\ref{hyp:3.1}.

\subsection{$G_\alpha\ne 1$ and $N$ is transitive on $V\Gamma$.}\label{sec:1}
By Hypothesis~\ref{hyp:3.1}, $(\Gamma,G)$ is a minimal counterexample. This minimality and the fact that $N$ is transitive on $V\Gamma$ imply that $G=N$. As $N$ is a minimal normal subgroup of $G$,  $G$ is simple. Thus $G\in \{\mathrm{Alt}(5),\mathrm{Alt}(6),\mathrm{PSp}(4,3)\}$. A computer computation in each of these cases shows that, if $G\le\mathrm{Aut}(\Gamma)$ has no semiregular elements of order at least $6$, then $|V\Gamma|\in \{10,20, 30,60, 90,180,360\}$, which contradicts Hypothesis~\ref{hyp:3.1}.

\subsection{$G_\alpha\ne 1$ and $N$ has two orbits on $V\Gamma$.}

Suppose $N$ is abelian. By~\cite[Lemma 1.15]{PS.fixedVertex}, either $\Gamma$ is complete bipartite, or $\Gamma$ is a bi-Cayley graph over $N$ and the minimal number of generators of $N$ is at most $4$. (Here, it is not really relevent to introduce the definition of bi-Cayley graph, however, what is really relevant is the fact that $N$ is generated by at most $4$ elements.) Recalling that $N$ is a $\{2,3,5\}$-group, it follows that $|V\Gamma|=2|N|\leq 2\cdot 5^4=1\,250$, and the equality is realized for $N=C_5^4$. In particular, this contradicts Hypothesis~\ref{hyp:3.1}.

Suppose $N$ is not abelian. By Lemma~\ref{lem:3.2},  $3$ divides the order of $N_\alpha$. A fortiori, $3$ divides the order of $G_\alpha$, hence $G$ acts arc-transitively on $\Gamma$. We can extract information on the local action of $G$ by consulting the amalgams in \cite[Section 4]{DjokovicMillerAmalgams}. In particular, with a direct inspection (on a case-by-case basis) on these amalgams, it can be verified that, for any edge $\{\alpha,\beta\}$ of $\Gamma$, $G$ contains an element $y$ that swaps $\alpha$ and $\beta$ and its order is either $2$ or $4$. As $\alpha$ and $\beta$ belong to distinct $N$-orbits, $y$ maps $\alpha^N$ to $\beta^N$. Moreover, as $N$ has two orbits on $V\Gamma$, the subgroup $N\langle y \rangle$ is vertex-transitive on $\Gamma$. Therefore, by minimality of $G$, we have $G=N\langle y \rangle$.

Assume $o(y)=2$. Thus $|G:N|=2$. As $N=S^l$ is a minimal normal subgroup of $G$, $l\in\{1,2\}$. If $l=1$, then $G$ is an almost simple group whose socle is either $\mathrm{Alt}(5)$, $\mathrm{Alt}(6)$ or $\mathrm{PSp(4,3)}$. A computer computation shows that $(\Gamma,G)$ satisfies Theorem~\ref{thm.main:1.2}, a contradiction. If $l=2$, then $\langle y\rangle$ permutes transitively the two simple direct factors of $N$. Let $s\in N$ be a $5$-element in a simple direct factor of $N$, and notice that $t:=s^y$ is a $5$-element in the other simple direct factor of $N$. Thus $[s,t]=1$. We claim that $ys$ is a semiregular element of order $10$. We get
\[ (ys)^2=ysys=ts \in N,\]
\[ (ys)^5=ysysysysys=ys(ts)^2 \in yN.\]
We have that $(ys)^2$ is a $5$-element in $N$, thus semiregular, and that $(ys)^5$ has order $2$ and, being an element of $yN=Ny$, it has no fixed points, hence it is semiregular. Therefore $ys$ is a semiregular element of order $10$, contradicting Hypothesis~\ref{hyp:3.1}.

Assume $o(y)=4$. As $|G:N|=4$ and $N$ is a minimal normal subgroup of $G$, $l\in \{1,2,4\}$. Observe that a Sylow $3$-subgroup of $G_\alpha$ has order $3$, because $\Gamma$ is cubic and $G$ is arc-transitive. Let $x\in G_\alpha$ be an element of order $3$. As $|G:N|=4$, we have $x\in N\cap G_\alpha=N_\alpha\le S^l$. In particular, we may write $x=(s_1,\ldots,s_l)$, with $s_i\in S$. Let $\kappa$ be the number of coordinates of $x$ different from $1$, we call $\kappa$ the type of $x$. Since $\langle x\rangle$ is a Sylow $3$-subgroup of $G_\alpha$, from Sylow's theorem, we deduce that each element of order $3$ in $G$ fixing some vertex of $\Gamma$ has type $\kappa$. Let $s\in S$ be an element of order $3$ and let $t\in S$ be an element of order $5$. Suppose $l=4$. If $\kappa\ne 1$, then $g=(s,t,1,1)$ has order $15$ and is semiregular because $g^5=(s^5,1,1,1)$ has order $3$ but it is not of type $\kappa$. Similarly, if $l=4$ and $k=1$, then $g=(s,s,t,1)$ has order $15$ and is semiregular. Analogously, when $l=2$, if $\kappa\ne 1$, then $g=(s,t)$ has order $15$ and is semiregular. When $l=2$, $\kappa=1$ and $S=\mathrm{PSp}(4,3)$, the group $S$ contains an element $r$ having order $9$ and hence $g=(r,r)$ is a semiregular element having order $9$. Summing up, from these reductions, we may suppose that either $l=1$, or $l=2$ and $S\in \{\mathrm{Alt}(5),\mathrm{Alt}(6)\}$. These cases can be dealt with a computer computation: indeed, the invaluable help of a computer shows that no counterexample to Theorem~\ref{thm.main:1.2} arises.

\subsection{$G_\alpha\ne 1$ and $\Gamma/N$ is a cycle of length $r\geq 3$.}The full automorphism group of $\Gamma/N$ is the dihedral group of order $2r$.
Let $K$ be the kernel of the action of $G$ on the $N$-orbits. The quotient $G/K$ acts faithfully on $\Gamma/N$, that is, it is a transitive subgroup of the dihedral group of order $2r$.

We claim that 
\begin{equation}\label{eq:O1}
	G/K\hbox{ is regular in its action on the vertices of }\Gamma/N.\end{equation} Assume $G/K$ acts on the vertices of  $\Gamma/N$ transitively but not regularly. In particular, $G/K$ is isomorphic to the dihedral group of order $2r$. Thus $G$ has an index $2$ subgroup $M$ such that $M$ is vertex-transitive and $M/K$ is isomorphic to the cyclic group of order $r$. By minimality of $G$, we have $G=M$, which goes against the choice of $M$. Hence $G/K$ is regular. In particular, either $G/K$ is isomorphic to the cyclic group of order $r$, or $r$ is even and $G$ is isomorphic to the dihedral group of order $r$. Later in this proof we resolve this ambiguity and we prove that $r$ is even and $G/K$ is dihedral of order $r$, see~\eqref{eq:OOOO1}.

As $G/K$ acts regularly on the vertices of $\Gamma/N$, we have
\[1_{G/K} = \left(\frac{G}{K}\right)_{\alpha^N} = \frac{G_\alpha K}{K}.\]
Therefore 
\begin{equation}\label{eq:bastaporcoboia}K_\alpha=K\cap G_\alpha=G_\alpha.\end{equation}	

Assume $G$ is arc-transitive. Let $\beta$ be a neighbour of $\alpha$ and observe that $\alpha^N\ne \beta^N$. Since $\Gamma$ is connected, we have
\[ G = \langle G_\alpha, G_{\{\alpha,\beta\}} \rangle = \langle K_\alpha, G_{\{\alpha,\beta\}} \rangle\le \langle K,G_{\{\alpha,\beta\}}\rangle = KG_{\{\alpha,\beta\}}, \]
and hence $G=KG_{\{\alpha,\beta\}}$.
Recalling that $K$ fixes all the $N$-orbits,
\[ |G:K| =|KG_{\{\alpha,\beta\}}:K|= |G_{\{\alpha,\beta\}}:K_{\{\alpha,\beta\}}| = |G_{\{\alpha,\beta\}}:G_{\alpha\beta}| = 2.\]
Thus $G/K\cong C_2$ and $r=2$, which is  a contradiction. Therefore
\begin{center}$G$ is not arc-transitive.\end{center}

This implies that $G_\alpha$ does not act transitively on the neighbourhood of $\alpha$, hence $G_\alpha$ is a $2$-group. By~\eqref{eq:bastaporcoboia}, we deduce $G_\alpha=K_\alpha$ is a $2$-group. Actually, Lemma~\ref{lemma:auxiliary} shows that 
\begin{equation}\label{eq:O2}G_\alpha=K_\alpha\hbox{ is an elementary abelian }2\hbox{-group}.\end{equation}

If $N$ is an elementary abelian $2$-group, then, by Lemma~\ref{lem:2.5}, $\Gamma$ is either a circular ladder, or a M\"{o}bius ladder, or a split Praeger-Xu graph $\mathrm{sC}(r/2,s)$. Now, in the former cases, the proof follows from Lemma~\ref{ladder1}, while, in the latter one, we conclude by Lemma~\ref{ex:PX}. In particular, for the rest of the proof we may suppose that $N$ is not an elementary abelian $2$-group.

For any minimal normal subgroup $M$ of $G$, $M_\alpha=M\cap G_\alpha$ is also a $2$-group. Thus, in view of Lemma~\ref{lem:3.2}, $M$ is an elementary abelian $p$-group, for some $p\in\{2,3,5\}$. This is true, in particular, for $N$. Let  $M$ be a minimal normal subgroup distinct from $N$. Since $[N,M]=1$, Lemma~\ref{lem:2.2} gives a contradiction unless $N$ and $M$ are both $p$-groups for the same prime $p$. Thus, 
\begin{equation}\label{eq:socle}\textrm{the socle of }G\textrm{ is an elementary abelian }p\textrm{-group, for some }p\in \{3,5\}.
\end{equation}

Before going any further, we need some extra information on the local action of $G$ on $\Gamma$. Since $G_\alpha$ is a non-identity $2$-group, there exists a unique vertex $\alpha'\in V\Gamma$ adjacent to $\alpha$ that is fixed by the action of $G_\alpha$. It follows that $\{\alpha,\alpha'\}$ is a block of imprimitivity for the action of $G$ on the vertices. Hence,
\[ G_\alpha \leq G_{\{\alpha,\alpha'\}} \quad \textup{and} \quad |G_{\{\alpha,\alpha'\}} : G_\alpha|=2.\]
We obtain that, for any $\beta\in V\Gamma$, neighbour of $\alpha$ distinct from $\alpha'$,
\[|G_{\{\alpha,\alpha'\}} : G_{\alpha\beta}|=4 \quad \textup{and} \quad |G_{\{\alpha,\beta\}} : G_{\alpha\beta}|=2.\]
Let $\{\alpha',\beta,\gamma\}$ be the neighbourhood of $\alpha$.

Assume $G/K$ is cyclic of order $r$. As $\Gamma/N$ is a cycle of length $r$, this means that $G/K$ acts transitively on the vertices and on the edges of $\Gamma/N$. Now, $\beta$ and $\gamma$ are in the same $K$-orbit because $K_\alpha=G_\alpha$ and $G_\alpha$ acts transitively on $\{\beta,\gamma\}$. In particular, each element in $\alpha^N$ has two neighbours in $\beta^N$. As $G/K$ is transitive on edges, we reach a contradiction because each element in $\alpha^N$ would have two neighbours in ${\alpha'}^N$, contradicting the fact that $\alpha$ has valency $3$. Thus
\begin{equation}\label{eq:OOOO1}
	r\hbox{ is even and }G/K\textrm{ is dihedral of order }r.
\end{equation}

Recall that $N$ is an elementary abelian $p$-group with $p\in\{3,5\}$. Thus $N$ is semiregular. We consider  $\mathbf{C}_K(N)$. Since $N\le {\bf C}_K(N)$ and since $K=K_\alpha N$, we deduce ${\bf C}_K(N)=N\times Q$, for some subgroup $Q$ of $K_\alpha$. As $K_\alpha$ is a $2$-group, so is $Q$. Therefore, $Q$ is characteristic in $N\times Q={\bf C}_K(N)$ and hence $Q\unlhd G$. Since $G_\alpha$ is a core-free subgroup of $G$, we get $Q=1$ and ${\bf C}_K(N)=N$.

Since $N$ is a minimal normal subgroup of $G$, $G$ acts irreducibly by conjugation on it, that is, $N$ is an irreducible $\mathbb{F}_pG$-module. As $K\unlhd G$, by Clifford's Theorem, $N$ is a completely reducible $\mathbb{F}_pK$-module. As $K=NG_\alpha$ and $N$ is abelian, $N$ is a completely reducible $\mathbb{F}_pG_\alpha$-module. As $G_\alpha$ is abelian, by Schur's Lemma, $G_\alpha$ induces on each irreducible $\mathbb{F}_pG_\alpha$-submodule a cyclic group action. However, since $G_\alpha$ has exponent $2$, we deduce that each irreducible $\mathbb{F}_pG_\alpha$-submodule has dimension $1$ and  $G_\alpha$ induces on each irreducible $\mathbb{F}_pG_\alpha$-submodule the scalars $\pm 1$. Therefore, $G_\alpha$ acts on $N$ by conjugation as a group of diagonal matrices having eigenvalues in $\{\pm1\}$. In other words, there exists a basis $(n_1,\ldots,n_e)$ of $N$ as a vector space over $\mathbb{F}_p$ such that, 
\begin{equation}\label{eq:OO1}
	\hbox{for each }g\in G_\alpha\hbox{ and for each }n_i,\hbox{ we have }n_i^g\in \{n_i,n_i^{-1}\}.\end{equation} Furthermore, the action of $G$ by conjugation on $N$ preserves the direct product decomposition $N=\langle n_1\rangle\times\cdots\times\langle n_e\rangle$.

We claim that
\begin{align}\label{eq:OOO1}
	{\bf C}_{G_{\{\alpha,\beta\}}}(N)&=1,\\\nonumber
	{\bf C}_{G_{\{\alpha,\alpha'\}}}(N)&=1.\nonumber
\end{align}	
In other words, $G_{\{\alpha,\beta\}}$ and $G_{\{\alpha,\alpha'\}}$ both act faithfully by conjugation on $N$. Let $\gamma\in \{\alpha',\beta\}$ and suppose, arguing by contradiction, that ${\bf C}_{G_{\{\alpha,\gamma\}}}(N)\ne 1$. Since ${\bf C}_K(N)=1$ and $|G_{\{\alpha,\gamma\}}:K\cap G_{\{\alpha,\gamma\}}|=2$, we deduce ${\bf C}_{G_{\{\alpha,\gamma\}}}(N)=\langle x\rangle$, where $x$ is an involution. Since $x\notin K$, $x$ acts semiregularly on $\Gamma/N$ and hence $x$ acts semiregularly on $\Gamma$. From this and from the fact that $x$ centralizes $N$, we deduce that $G$ contains semiregular elements of order $2p\ge 6$, which contradicts Hypothesis~\ref{hyp:3.1}. Thus~\eqref{eq:OOO1} is proven.

Observe that~\eqref{eq:OOO1} implies that an element of $G_{\{\alpha,\alpha'\}}$ or of $G_{\{\alpha,\beta\}}$ is the identity if and only it its action on $N$ by conjugation is trivial.

We show that 
\begin{equation}\label{eq:O3}G_{\{\alpha,\beta\}} \setminus G_{\alpha\beta}\hbox{ contains an involution.}\end{equation}

Let $H$ be the permutation group induced by $G_{\{\alpha,\alpha'\}}$ in its action on the four right cosets  of $G_{\alpha\beta}$ in $G_{\{\alpha,\alpha'\}}$. Since $H$ is a $2$-group, $H$ is isomorphic to either $C_4$, or $C_2\times C_2$, or to the dihedral group of order $8$. In the first two cases, $G_{\alpha\beta}$ is a normal subgroup of both $G_{\{\alpha,\alpha'\}}$ and $G_{\{\alpha,\beta\}}$. As $G_{\alpha\beta}$ is core-free in $G$ and
\[G = \langle G_{\{\alpha,\alpha'\}}, G_{\{\alpha,\beta\}}\rangle, \]
we have that $G_{\alpha\beta}=1$. In particular, $G_{\{\alpha,\beta\}}$ is cyclic of order $2$, hence it contains an involution and~\eqref{eq:O3} follows in this case.

In the latter case, using the notation and the terminology in~\cite{DjokovicAmalgams4:2}, we have that the triple $(G_{\{\alpha,\alpha'\}},G_{\alpha\beta},G_{\{\alpha\,\beta\}})$ is a locally dihedral faithful group amalgam of type $(4,2)$ and $G$ is one of its realizations. Indeed, from the classification in~\cite{DjokovicAmalgams4:2}, we see that  either $G_{\{\alpha,\alpha'\}}\setminus G_\alpha$ or $G_{\{\alpha,\beta\}}\setminus G_{\alpha\beta}$ contains an involution. If $G_{\{\alpha,\beta\}}\setminus G_{\alpha\beta}$ contains an involution, then~\eqref{eq:O3} holds true also in this case. Therefore we   suppose $\tau_1 \in G_{\{\alpha,\alpha'\}}\setminus G_\alpha$ is an involution. We investigate the action by conjugation of $\tau_1$ on $N$.  By~\eqref{eq:O1}, $\tau_1$ is a semiregular automorphism of $\Gamma/K$, because $\tau_1\notin K$. Therefore, $\tau_1$ is a semiregular automorphism of $\Gamma$.
Since no semiregular involution commutes with a non-identity element of $N$, $\tau_1$ acts  by conjugation on $N$ without fixed points, that is, for any $n\in N$, $n^{\tau_1}=n^{-1}$. It follows from~\eqref{eq:OO1} that  $\tau_1$ commutes with $G_\alpha$ and hence $G_{\{\alpha,\alpha'\}}=\langle G_\alpha, \tau_1 \rangle$ is an elementary abelian $2$-group. Now, as $G_{{\alpha\beta}}$ is normal in both $G_{\{\alpha,\alpha'\}}$ and $G_{\{\alpha,\beta\}}$, we can conclude, as before, that $G_{\{\alpha,\beta\}}$ is cyclic of order $2$, hence it contains an involution.
Therefore, in any case,~\eqref{eq:O3} holds true.

Let $e$ be the positive integer such that $N=C_p^e$. We aim to show that 
\begin{equation}\label{eq:O0}e\in \{1,2\}.\end{equation} 	
Let $\tau_2 \in G_{\{\alpha,\beta\}}\setminus G_{\alpha\beta}$ be an involution: the existence of $\tau_2$ is guaranteed by~\eqref{eq:O3}. Now, we look at the action by conjugation of $\tau_2$ on $N$. Observe $\tau_2\notin K$ and hence $\tau_2$ is a semiregular automorphism of $\Gamma$. Therefore, arguing as in the previous paragraph (with the involution $\tau_1$ replaced by $\tau_2$), we deduce that $n^{\tau_2}=n^{-1}$ for every $n\in N$. Let $L:=\langle \tau_2^g\mid g\in G\rangle$. Since $G/K$ is a dihedral group and $\tau_2$ is an involution, we deduce that $|G/K:LK/K|\le 2$, that is, $|G:LK|\le 2$. Observe now that, for any $n \in N$, $n^{\tau_2^g}=n^{-1}$. Therefore, the group induced by the action by conjugation of $L$ on $N$ has order $2$. This and~\eqref{eq:OO1} shows that the subgroup $LK$ of $G$ preserves the direct sum decomposition $N=\langle n_1\rangle\times\cdots\times \langle n_e\rangle$. However, since $G$ acts irreducibly on $N$ and since $|G:LK|\le 2$, we finally obtain $e\le 2$, as claimed in~\eqref{eq:O0}.  Observe that from this it follows that $|N|=p^e\in \{3,9,5,25\}$.

We are now ready to conclude this case. Observe that $G_\alpha$ contains an element $x$ with $n^x=n^{-1}$ for every $n\in N$. This is immediate from~\eqref{eq:OO1} when $e=1$, or when $e=2$ and $|G_\alpha|=4$. When $e=2$ and $|G_\alpha|<4$, we have $|G_\alpha|=2$ and hence the non-identity element of $G_\alpha$ acts by conjugation on $N$ inverting each of its elements. 

Now, $x$ and $\tau_2$ both induce the same action by conjugation on $N$, contradicting~\eqref{eq:OOO1}. This final contradiction has concluded the analysis of this case.

\subsection{$G_\alpha\ne 1$ and $\Gamma/N$ is a cubic graph.}

Under this assumption, any two distinct neighbours of $\alpha$ are in distinct $N$-orbits, thus  $N_\alpha =1$. In particular, Lemma~\ref{lem:3.2} gives that $N$ is elementary abelian. Set $\bar{\Gamma}:=\Gamma/N$, $\bar G:=G/N$ and $\bar\alpha:=\alpha^N$. Since $|V\bar\Gamma|<|V\Gamma|$, by Hypothesis~\ref{hyp:3.1} the pair $(\bar \Gamma,\bar G)$ is not a counterexample to Theorem~\ref{thm.main:1.2} and hence $(\bar \Gamma,\bar G)$ is one of the pairs appearing in Table~\ref{table:table}. Moreover, since $G_\alpha\ne 1$, we have the additional information that a vertex-stabilizer $\bar{G}_{\bar\alpha}\cong G_\alpha$ is not the identity.

We have resolved this case with a computer computation. Since this computer computation is quite involved, we give some details. Let $(\bar \Gamma,\bar G)$ be any pair in Table~\ref{table:table}, except for the last row. For each prime $p\in \{2,3,5\}$, we have constructed all the irreducible modules  of $\bar G$ over the field $\mathbb{F}_p$ having $p$ elements. Let $V$ be one of these irreducible modules. This module $V$ corresponds to the putative minimal normal subgroup $N$ of $G$.  
We have constructed all the distinct extensions of $\bar G$ via $V$. Let $E$ be one of these extensions and let $\pi:E\to \bar G$ be the natural projection with $\mathrm{Ker}(\pi)=V$. This extension $E$ corresponds to the putative abstract group $G$.  For each such extension, we have computed all the subgroups $H$ of $E$ with the property that $\pi_{|H}$ is an isomorphism between $H$ and $\bar{G}_{\bar\alpha}$. This subgroup $H$ is our putative vertex-stabilizer $G_\alpha$. This computation can be performed in $\pi^{-1}(\bar{G}_{\bar{\alpha}})$. Next, we have constructed the permutation representation $E_p$ of $E$ acting on the right cosets of $H$ in $E$. This permutation group $E_p$ is our putative permutation group $G$. If $E_p$ has semiregular elements of order at least $6$, then we have discarded $E$ from further consideration.

For each permutation group $E_p$ as above, we have verified, by considering the orbital graphs of $E_p$, whether $E_p$ acts on a connected cubic graph. This is our putative graph $\Gamma$. This step is by far the most expensive step in the computation.	

This whole process had to be applied repeatedly starting with the pairs arising from the census of connected cubic graphs having at most $1\,280$ vertices.

For instance, the graphs having $65\,610$ vertices were found by applying this procedure starting with the graph having $810$ vertices and its transitive group of automorphisms having $1\,620$ elements: here the elementary abelian cover $N$ has cardinality $81=3^4$. Incidentally, we have found only one pair up to isomorphism. Next, by applying this procedure to this pair, we found no new examples.

We give some further details of the computation when we applied the procedure with $\bar \Gamma$ having $1\,250=2\cdot 5^4$ vertices and with its corresponding vertex-transitive subgroup $\bar G$ having order $2\,500=2^2\cdot 5^4$. When we applied this procedure, we have obtained graphs having $2\cdot 5^5=6\,250$ vertices and admitting a group of automorphisms having $2^2\cdot 5^5=12\,500$ elements. Actually, in this step, we have found only one pair up to isomorphism. We have repeated this procedure two more times, obtaining graphs having $2\cdot 5^6=31\,250$ and $2\cdot 5^7=156\,250$ vertices. We were not able to push this computation further. Therefore to complete the proof of Theorem~\ref{thm.main:1.2}, we need to show that any new pair $(\Gamma,G)$ has the property that $|V\Gamma|=2\cdot 5^\ell$ and $|G|=4\cdot 5^\ell$, with $\ell\le 34$.

From the discussion above we may suppose that $|V\bar\Gamma|=2\cdot 5^\ell$ and $|\bar G|=4\cdot 5^\ell$ with $\ell\le 34$. Moreover, $\bar \Gamma$ is a regular cover of the graph, say $\Delta$, having $1\,250$ vertices and $\bar{G}$ is a quotient of the group of automorphisms of $\Delta$, say $H$, with $|H|=2\,500$. In particular, a Sylow $2$-subgroup of $\bar{G}$ is cyclic and $\bar{G}$ has a normal Sylow $5$-subgroup. (This information can be extracted from the analogous properties of $H$.) Let $\bar{P}$ be a Sylow $5$-subgroup of $\bar G$ and observe that every non-identity element of $\bar P$ has order $5$ because every semiregular element of  $\bar G$ has order at most $6$. Let $P$ be the subgroup of $G$ with $G/N=\bar P$. Assume $N$ is not an elementary abelian $5$-group. Then $N$ is an elementary abelian $p$-group for some $p\in \{2,3\}$. Let $Q$ be a Sylow $5$-subgroup of $P$ and observe that $P=N\rtimes Q$. The elements in $P$ are semiregular and hence each element of $P$ has order at most $6$. This implies that the elements of $P$ have order $1$, $5$ or $p$. This implies that the action, by conjugation, of $Q$ on $N$ is fixed-point-free and $P$ is a Frobenius group with Frobenius kernel $N$ and Frobenius complement $Q$. The structure theorem of Frobenius complements gives that $Q$ is cyclic and hence $|Q|=5$, which is a contradiction. This contradiction has shown that $N$ is an elementary abelian $5$-group and hence $P$ is a Sylow $5$-subgroup of $G$. Moreover, $G=P\rtimes \langle x\rangle$, where $\langle x\rangle$ is a cyclic group of order $4$. We have shown that $|V\Gamma|=2\cdot 5^{\ell'}$ and $|G|=2^2\cdot 5^{\ell'}$. Therefore, it remains to show that $\ell'\le 34$.

Since $|G_\alpha|=2$, $G_\alpha$ fixes a unique neighbour of $\alpha$. Let us call $\alpha'$ this neighbour. Now, $G_{\{\alpha,\alpha'\}}$ has order $4$ because $\{\alpha,\alpha'\}$ is a block of imprimitivity for the action of $G$ on $V\Gamma$. Therefore, by Sylow's theorem, we may suppose that
$$G_{\{\alpha,\alpha'\}}=\langle x\rangle.$$
In particular, $G_\alpha=\langle x^2\rangle$.

Let $\beta$ and $\gamma$ be the neighbours of $\alpha$ with $\beta\ne \alpha'\ne \gamma$. Clearly, $|G_{\{\alpha,\beta\}}|=2$ and hence, by Sylow's theorem, $$G_{\{\alpha,\beta\}}=\langle (x^2)^y\rangle,$$ for some $y\in P$. 

Since $\Gamma$ is connected, we have
$$G=\langle G_{\{\alpha,\alpha'\}},G_{\{\alpha,\beta\}}\rangle=\langle x,(x^2)^y\rangle=\langle x,y^{-1}y^{x^2}\rangle.$$
As $P\unlhd G$ and $o(x)=4$, we deduce
$$P=\langle y^{-1}y^{x^2},(y^{-1}y^{x^2})^x,(y^{-1}y^{x^2})^{x^2},(y^{-1}y^{x^2})^{x^3}\rangle.$$
Now,
$$(y^{-1}y^{x^2})^{x^2}=(y^{x^2})^{-1}y^{x^4}=(y^{x^2})^{-1}y=(y^{-1}y^{x^2})^{-1}.$$
Therefore, $P=\langle y^{-1}y^{x^2},(y^{-1}y^{x^2})^x\rangle$ is a $2$-generated group of exponent $5$. In view of the restricted  Burnside problem (see \cite{HWW1974} and \cite{Zelmanov}), the order of $P$ is at most $5^{34}$ and hence $\ell'\le 34$.

\bibliographystyle{alpha}
\bibliography{semiregularBib}
\end{document}